\documentclass[11pt]{amsart}
\usepackage{amssymb,xspace,hyperref}

\PassOptionsToPackage{hyphens}{url}\usepackage{hyperref}

\makeatletter
\g@addto@macro{\UrlBreaks}{\UrlOrds}
\makeatother
\usepackage{color}
\definecolor{darkred}{rgb}{0.4,0,0}
\definecolor{darkgreen}{rgb}{0,0.5,0}
\definecolor{darkblue}{rgb}{0,0,0.4}
\hypersetup{
	pdftitle={Smooth and Rough Positive Currents},	% title
	pdfauthor={S. Filip, V. Tosatti},	% author
	pdfsubject={Smooth and Rough Positive Currents},		% subject of the document
	pdfkeywords={Kahler currents},	% list of keywords
	pdfnewwindow=true,		% links in new window
	colorlinks=true,		% false: boxed links; true: colored links
	linkcolor=darkblue,		% color of internal links (change box color with linkbordercolor)
	citecolor=darkred,		% color of links to bibliography
	filecolor=darkblue,		% color of file links
	urlcolor=darkblue,		% color of external links
	pdfborder={0 0 0},
	breaklinks=true
}
\title{Smooth and Rough Positive Currents}

\dedicatory{Dedicated to Professor Jean-Pierre Demailly on the occasion of his 60th birthday}

\author{Simion Filip}
\address{
\parbox{0.5\textwidth}{
	Department of Mathematics\\
	Harvard University\\
	1 Oxford St,
	Cambridge, MA 02138}
}
\email{{sfilip@math.harvard.edu}}

\author{Valentino Tosatti}
\address{
\parbox{0.5\textwidth}{
	Department of Mathematics\\
	Northwestern University\\
	2033 Sheridan Road,
	Evanston, IL 60208}
}
\email{tosatti@math.northwestern.edu}

\theoremstyle{plain}
\newtheorem{theorem}{Theorem}[section]
\newtheorem{proposition}[theorem]{Proposition}
\newtheorem{definition}[theorem]{Definition}

\newtheorem{problem}[theorem]{Problem}

\newtheorem{conjecture}[theorem]{Conjecture}
\newtheorem{question}[theorem]{Question}

\theoremstyle{definition}
\newtheorem{remark}[theorem]{Remark}
\theoremstyle{remark}

\numberwithin{equation}{section}

\newcommand{\Kahler}{{{K}\"ahler}\xspace}
\newcommand{\dVol}{\operatorname{dVol}}
\newcommand{\norm}[1]{\left\|#1\right\|}

\newcommand{\Null}{\textnormal{Null}}
\newcommand{\del}{\partial}
\newcommand{\de}{\partial}
\newcommand{\dbar}{\overline{\del}}
\newcommand{\ddbar}{i\del\dbar}
\newcommand{\db}{\overline{\del}}

\newcommand{\ve}{\varepsilon}
\newcommand{\vp}{\varphi}
\newcommand{\ti}[1]{\tilde{#1}}

\renewcommand{\leq}{\leqslant}
\renewcommand{\geq}{\geqslant}

\begin{document}

\begin{abstract}
We study the different notions of semipositivity for $(1,1)$ cohomology classes on $K3$ surfaces. We first show that every big and nef class (and every nef and rational class) is semiample, and in particular it contains a smooth semipositive representative.
By contrast, we show that there exist irrational nef classes with no closed positive current representative which is smooth outside a proper analytic subset.
We use this to answer negatively two questions of the second-named author. Using a result of Cantat \& Dupont, we also construct examples of projective $K3$ surfaces with a nef $\mathbb{R}$-divisor which is not semipositive.
\end{abstract}

\maketitle

\section{Introduction}
The Kodaira Embedding Theorem shows that a holomorphic line bundle on a compact complex manifold is positive (i.e. admits a smooth Hermitian metric with strictly positive curvature form) if and only if it is ample (the algebro-geometric notion of strict positivity for line bundles). In contrast to this, there are several natural notions of semipositivity for line bundles, which in general are not equivalent.

To be precise, let $X$ be an $n$-dimensional compact complex manifold, equipped with a Hermitian metric $\omega$, and $L$ a holomorphic line bundle on $X$. The most common notions of semipositivity for $L$ are:

\begin{itemize}
\item[(1)] $L$ is {\em semiample} if there is $m\geq 1$ such that $L^{\otimes m}$ is globally generated.
\item[(2)] $L$ is {\em Hermitian semipositive} if there is a smooth Hermitian metric $h$ on $L$ with curvature form $R_h$ which is semipositive definite on $X$.
\item[(3)] $L$ is {\em nef} if for all $\ve>0$ there is a smooth Hermitian metric $h_\ve$ on $L$ with curvature $R_{h_\ve}\geq -\ve\omega$ on $X$.
\item[(4)] $L$ is {\em pseudoeffective} if there is a singular Hermitian metric $h$ on $L$ with curvature current $R_h\geq 0$ on $X$ in the weak sense.
\end{itemize}
If $X$ is K\"ahler then nefness is equivalent to $c_1(L)$ being in the closure of the K\"ahler cone, and if $X$ is furthermore projective then it is also equivalent to $(L\cdot C)\geq 0$ for all curves $C\subset X$ (see e.g. \cite{De2}).

It is well-known \cite{De2} that $(1)\Rightarrow (2)\Rightarrow (3)\Rightarrow (4)$, and there are examples that show that all these implications are strict. For $(4)\not\Rightarrow (3)$ one can take $X$ to be the blowup of $\mathbb{P}^2$ at a point and $L=\mathcal{O}(E)$ where $E$ is the exceptional divisor. The first example where $(3)\not\Rightarrow (2)$ was discovered by Demailly--Peternell--Schneider \cite{DPS}, answering negatively a conjecture of Fujita \cite{Fu}. As for $(2)\not\Rightarrow (1)$, note first of all that being semiample is not a numerical condition, so a ``trivial'' example is a non-torsion line bundle $L\in \mathrm{Pic}^0(X)$ (with $X$ e.g. a torus). But there are also examples of $L$ Hermitian semipositive which is not numerically equivalent to a semiample line bundle: one can take the famous example of Zariski \cite[2.3.A]{Laz}, which was recently shown to be Hermitian semipositive by Koike \cite{Koi}.

However, if $X$ happens to be Calabi-Yau (i.e. compact K\"ahler with torsion canonical bundle) and projective then Kawamata-Shokurov's base-point-free theorem \cite{KMM} shows that if $L$ is nef and big (which for nef line bundles means $(L^n\cdot X)>0$) then $L$ is in fact semiample. In fact, this even holds for a nef and big $\mathbb{R}$-divisor $D$, where $D$ being semiample now means that there is a morphism $f:X\to Y$ onto a normal projective variety such that $D$ is $\mathbb{R}$-linearly equivalent to the pullback of an ample $\mathbb{R}$-divisor on $Y$, and is proved in \cite[Theorem 3.9.1]{BCHM}.

Dropping the assumption of bigness, one obtains the following well-known open problem:

\begin{conjecture}\label{bpf2}
Let $X$ be a projective manifold with torsion canonical bundle, and $L$ a nef line bundle on $X$. Then $L$ is numerically equivalent to a semiample line bundle.
\end{conjecture}

This conjecture is in fact a consequence of the log abundance conjecture, provided one can show that $L$ is numerically equivalent to a $\mathbb{Q}$-effective line bundle (a ``nonvanishing'' type conjecture), and is open even in dimension $3$ (see e.g. \cite{Wi,LOP} and references therein). Of course, Conjecture \ref{bpf2} would imply the same statement for nef $\mathbb{Q}$-divisors, but now unlike the nef and big case, one cannot expect this to be true for nef $\mathbb{R}$-divisors: if the class $c_1(D)$ of a nef $\mathbb{R}$-divisor $D$ cannot be written as a positive real linear combination of classes of nef $\mathbb{Q}$-divisors, then $D$ cannot be numerically equivalent to a semiample $\mathbb{R}$-divisor. Explicit examples are given in Section \ref{rough}.

Going back to the general setting of a compact complex $n$-manifold $X$, if $\alpha$ is a closed real $(1,1)$ form on $X$, then one can consider its class $[\alpha]$ in the Bott-Chern cohomology $H^{1,1}(X,\mathbb{R})$ of closed real $(1,1)$ forms modulo $\ddbar$-exact ones; in other words, the representatives of the $(1,1)$ class $[\alpha]$ are of the form $\alpha+\ddbar\vp,\vp\in C^\infty(X,\mathbb{R})$. If $[\alpha]=c_1(L)$ for some holomorphic line bundle $L$, then these representatives can be identified with curvature forms of smooth Hermitian metrics on $L$. In particular, one can extend the semipositivity notions above to $(1,1)$ classes:

\begin{itemize}
\item[(1)] $[\alpha]$ is {\em semiample} if there is a holomorphic surjective map $f:X\to Y$ onto a normal compact K\"ahler analytic space such that $[\alpha]=[f^*\beta]$ for some K\"ahler metric $\beta$ on $Y$.
\item[(2)] $[\alpha]$ is {\em semipositive} if it contains a smooth semipositive representative $\alpha+\ddbar\vp\geq 0$.
\item[(3)] $[\alpha]$ is {\em nef} if for every $\ve>0$ there is a representative which satisfies $\alpha+\ddbar\vp_\ve\geq -\ve\omega$.
\item[(4)] $[\alpha]$ is {\em pseudoeffective} if it contains a closed positive current $\alpha+\ddbar\vp\geq 0$ in the weak sense, where $\vp$ is quasi-psh.
\end{itemize}

Of course we still have that $(1)\Rightarrow (2)\Rightarrow (3)\Rightarrow (4)$, and again we will say that a nef $(1,1)$ class $[\alpha]$ is also big if $\int_X \alpha^n>0$.

The transcendental version of Kawamata-Shokurov's base-point-free theorem is the following (see also \cite[Conjecture 6.1]{To} for the case of Calabi-Yau manifolds as well as \cite[Conjectures 4.13 and 4.17]{To5} for weaker versions):
\begin{conjecture}\label{bpf}
Let $X$ be a compact K\"ahler manifold and $[\alpha]$ a nef $(1,1)$ class on $X$ such that $\lambda[\alpha]-c_1(K_X)$ is nef and big for some $\lambda>0$. Then $[\alpha]$ is semiample.
\end{conjecture}

In this paper we will focus on compact complex surfaces. First, we observe that in this case Conjecture \ref{bpf} holds:

\begin{theorem}\label{big}
Conjecture \ref{bpf} holds when $\dim X=2$.
\end{theorem}
The proof of this result uses some ideas from \cite{CT}. Despite recent advances in the Minimal Model Program for K\"ahler $3$-folds (see \cite{HP}), Conjecture \ref{bpf} remains open in dimensions $3$ or higher. However, the work of \cite{HP} can be used to prove Conjecture \ref{bpf} for $3$-folds in many cases, such as when $\lambda[\alpha]-c_1(K_X)$ is K\"ahler and the extremal face in the cone of classes of closed positive $(2,2)$ currents which intersect $[\alpha]$ trivially is in fact an extremal ray (cf. the discussion in \cite{TZ} after Conjecture 1.2). Also, in dimension $3$ if $X$ is assumed to have torsion canonical bundle, then Conjecture \ref{bpf} holds, as was communicated to us by Andreas H\"oring \cite{Hor}: a finite \'etale cover of $X$ is then either a torus, a projective Calabi-Yau manifold with no holomorphic $2$-forms, or the product of an elliptic curve and a $K3$ surface. The first case is easy, in the second case Kawamata-Shokurov's base-point-free theorem applies, and the third case can be dealt with. Furthermore, H\"oring shows in \cite{Hor} that Conjecture \ref{bpf} holds in dimension $3$ if $\lambda[\alpha]-c_1(K_X)$ is K\"ahler.

One may now wonder what happens to nef $(1,1)$ classes which are not big. As mentioned in Conjecture \ref{bpf2} above, on projective Calabi-Yau manifolds, nef line bundles are expected to be semiample, after possibly twisting by a numerically trivial line bundle. Of course this would imply the same statement for nef $\mathbb{Q}$-divisors. We record here the well-known fact that this holds in the case of $K3$ surfaces:
\begin{proposition}\label{nef}
Let $X$ be a $K3$ surface and $[\alpha]$ a nef $(1,1)$ class on $X$ with $[\alpha]\in H^2(X,\mathbb{Q})$. Then $[\alpha]$ is semiample.
\end{proposition}

In light of Theorem \ref{big} and Proposition \ref{nef}, the following result may therefore come as a surprise:

\begin{theorem}\label{main}
We have that
\begin{itemize}
\item[(a)] There is a non-projective $K3$ surface with a nef $(1,1)$ class $[\alpha]$ (which is necessarily not big and not rational) which is not semipositive.
\item[(b)] There is a projective $K3$ surface with a nef $\mathbb{R}$-divisor $D$ which is not Hermitian semipositive. In particular, $D$ cannot be numerically equivalent to a semiample $\mathbb{R}$-divisor.
\end{itemize}
\end{theorem}
Part (a) gives a counterexample to a question of the second-named author in \cite{To4}, and can also be used to provide a counterexample to a related question in \cite{To3,To4}, see Questions \ref{q1} and \ref{q2} below, as well as Theorem \ref{negative}.

Part (b) should be compared with Conjecture \ref{bpf2}. Furthermore, the nef class $c_1(D)$ in Theorem \ref{main} (b) is extremal in the nef cone, and its line is not defined over $\mathbb{Q}$ (these properties follow from \cite[Remark 1.1 and Lemma 1.3]{Cantat_K3}), which implies that it cannot be written as a positive real linear combination of classes of nef $\mathbb{Q}$-divisors.

The proof of Theorem \ref{main} uses holomorphic dynamics. For part (a), we will take $X$ to be one of the non-projective $K3$ surfaces constructed by McMullen \cite{McM} which have an automorphism with a Siegel disc. Using work of Cantat \cite{Cantat_K3}, we will show that at least one of the two ``eigenclasses'' associated to this automorphism cannot have a smooth semipositive representative, otherwise the Siegel disc would necessarily have Lebesgue measure zero.

Since McMullen's examples with a Siegel disc are necessarily non-projective, part (b) requires more sophisticated tools from dynamics, and specifically a crucial result of Cantat-Dupont \cite{CD}. In fact in this case we prove a more precise result than this. The surface $X$ is a generic hypersurface in $\mathbb{P}^1\times\mathbb{P}^1\times\mathbb{P}^1$ of degree $(2,2,2)$, and the class $[\alpha]$ is an eigenvector for the action on $H^{1,1}(X,\mathbb{R})$ induced by a chaotic automorphism of $X$, following Cantat \cite{Cantat_K3}. He proved that this class, as well as the corresponding class for the inverse automorphism, contain a unique closed positive current, and we show that both of these currents are not smooth at any point in the support of the measure with maximal entropy, see Theorem \ref{thm:main} below. \\

{\bf Acknowledgments. }The authors are grateful to J.-P. Demailly for his seminal work and his mathematical vision, which are an inspiration for much of our work.
We also thank A. H\"oring for useful communications in \cite{Hor} about Conjecture \ref{bpf}, C. Xu and T. Collins for discussions, M. Verbitsky for pointing out his related ongoing work with N. Sibony extending the uniqueness in Theorem \ref{thm:main} (i) to more irrational nef $(1,1)$ classes with volume zero on $K3$ surfaces, S. Takayama and T. Koike for discussions about Problem \ref{prob}, N. Sibony for providing references on rigidity, and the referee for useful comments including Remark \ref{refe}.
This research was partially conducted during the period the first-named author served as a Clay Research Fellow. The second-named author was partially supported by NSF grant DMS-1610278, and this work was finalized during his visit to the Center for Mathematical Sciences and Applications at Harvard University, which he would like to thank for the hospitality.

\section{Smooth positive currents}
In this section we give the proof of Theorem \ref{big} and Proposition \ref{nef}.

\begin{proof}[Proof of Theorem \ref{big}]
Up to renaming the class $[\alpha],$ we may assume that $\lambda=1$. First assume that the canonical bundle $K_X$ of $X$ is not pseudoeffective. This implies that $H^{2,0}(X)=H^0(X,K_X)=0$, which in turn implies that $X$ is projective by a well-known result of Kodaira, and it also implies that $[\alpha]=c_1(D)$ for some nef $\mathbb{R}$-divisor $D$ on $X$. We can thus conclude by the base-point-free theorem for $\mathbb{R}$-divisors \cite[Theorem 3.9.1]{BCHM}.

We can therefore assume that $K_X$ is pseudoeffective. Then the nef class $[\alpha]=[\alpha]-c_1(K_X)+c_1(K_X)$ is a sum of big and pseudoeffective, hence big.
Thanks to a result of Lamari \cite{Lam} (which was generalized by Demailly-P\u{a}un \cite{DP} to all dimensions), the class $[\alpha]$ contains a K\"ahler current $T=\alpha+\ddbar\psi$, i.e. a closed positive current which satisfies $T\geq\ve\omega$ in the weak sense, for some $\ve>0$, where $\omega$ is a K\"ahler form on $X$. If we consider all such K\"ahler currents and take the intersection of the subsets of $X$ where these currents are not smooth, we obtain a closed proper analytic subset $E_{nK}(\alpha)\subset X$, the non-K\"ahler locus of $[\alpha]$ (see \cite{Bou}). We may assume that $E_{nK}(\alpha)$ is nonempty, otherwise $[\alpha]$ is K\"ahler \cite{Bou} and the theorem is trivial. Using Demailly's regularization \cite{De} and an observation of Boucksom \cite[Theorem 3.17 (ii)]{Bou}, we can find one K\"ahler current $T=\alpha+\ddbar\psi$ such that $\psi$ is singular precisely along $E_{nK}(\alpha)$.
By the main result of \cite{CT}, we have that $E_{nK}(\alpha)=\Null(\alpha),$ where the null locus $\Null(\alpha)$ is defined to be the union of all irreducible curves $C\subset X$ such that $\int_C\alpha=0$ (in fact in dimension $2$ the proof simplifies greatly, see \cite[Section 5]{To2}). In particular $\Null(\alpha)$ is itself a proper analytic subvariety of $X$, of pure dimension $1$, with irreducible components given by curves $C_i, i=1,\dots N,$ with the intersection matrix $(C_i\cdot C_j)$ negative definite. Indeed for any real numbers $\lambda_i$ we have $\int_X\alpha^2>0$ and $\int_{\sum_i \lambda_i C_i}\alpha=0$, so by the Hodge index theorem \cite[Corollary IV.2.16]{BHPV} either $\sum_i \lambda_i C_i=0$ or $(\sum_i \lambda_i C_i)^2<0$. This means that the intersection form is negative definite on the linear span of the curves $C_i$, which proves our assertion. Grauert's criterion \cite[Theorem III.2.1]{BHPV},  \cite[p.367]{Gr}, shows that each connected component of $\Null(\alpha)$ can be contracted, so we get a holomorphic map $\pi:X\to Y$ where $Y$ is an irreducible normal compact complex surface, and $\pi$ is an isomorphism away from $\Null(\alpha)$, and it contracts each connected component of $\Null(\alpha)$ to a point in $Y$. Let $S=\cup_{j=1}^k\{p_j\}$ be the image $\pi(\Null(\alpha))$, so $Y$ is smooth away from $S$.

In general the problem of deciding whether a normal surface $Y$ is projective or K\"ahler is nontrivial, and has a long history (see e.g. \cite{Ar, Fuj, Gr,Moi, Za}). In our case, we use the key fact that $[\alpha]-c_1(K_X)$ is nef (and big) to obtain that for all $i$,
$$0\leq \int_{C_i}([\alpha]-c_1(K_X))=-K_X\cdot C_i.$$
Combining this with $(C_i^2)<0$ and with the adjunction formula
$$p_a(C_i)=1+\frac{(K_X\cdot C_i)+(C_i^2)}{2}\leq\frac{1}{2},$$
to conclude that $p_a(C_i)=0$ and so each $C_i$ is a smooth rational curve. This in turn shows that either $K_X\cdot C_i=0$, in which case $(C_i^2)=-2$ and $C_i$ is therefore a $(-2)$-curve, or else $K_X\cdot C_i=-1$, in which case $(C_i^2)=-1$ and $C_i$ is a $(-1)$-curve.

We now claim that all singular points of $Y$ are rational singularities. To see this we use Artin's criterion \cite[Theorem III.3.2]{BHPV}, and so it suffices to check that for each connected component $\bigcup_{i\in I}C_i$ of $\Null(\alpha)$ (where $I\subset\{1,\dots,N\}$) and for every integers $r_i\geq 0, i\in I$ (not all zero), the divisor $Z=\sum_{i\in I}r_i C_i$ satisfies $p_a(Z)\leq 0$. But this follows from
$$p_a(Z)=1+\frac{(Z^2)+\sum_{i\in I}r_i (K_X\cdot C_i)}{2} < 1+ \frac{1}{2} \sum_{i\in I} r_i (K_X\cdot C_i)\leq 1,$$
using that $(Z^2)<0$ since $\int_Z\alpha=0$.

By \cite[Lemma 3.3]{HP}, there is a smooth closed real $(1,1)$ form $\beta$ on $Y$, which is locally $\de\db$-exact, such that $[\pi^*\beta]=[\alpha]$. The pushforward $\pi_* T$ is then a closed positive $(1,1)$ current on $Y$, thanks to \cite[p.17]{Dem85}, and $\pi_*T$ is a smooth K\"ahler metric on $Y\backslash S$. Since the pullback of smooth forms on $Y$ give smooth forms on $X$, we can easily check that $\pi_*T$ is a K\"ahler current on $Y$.
By \cite[Lemma 3.4]{HP}, $\pi_*T$ is locally $\de\db$-exact, and it lies in the class $[\beta]$ in $H^{1,1}(Y,\mathbb{R})$, so we can write $\pi_*T=\beta+\ddbar\ti{\psi}$ for some quasi-psh function $\ti{\psi}$ on $Y$ which is smooth away from $S$. Near each singular point of $Y$, we choose an embedding of a neighborhood of the singular point as the unit ball in $\mathbb{C}^N$ and let
$$\hat{\psi}(z)=\widetilde{\max}(\ti{\psi}(z),A|z|^2-C),$$
where $\widetilde{\max}$ is a regularized maximum (see e.g. \cite[I.5.18]{Demb}) and we first choose $A$ large so that $\beta+A\ddbar |z|^2$ is a K\"ahler metric on this ball, and then we choose $C$ large so that $\hat{\psi}=\ti{\psi}$ in a neighborhood of the boundary of this ball ball. We can then define $\hat{\psi}=\ti{\psi}$ also outside of this ball, and after repeating this procedure at all singular points of $Y$, we obtain that $\beta+\ddbar\hat{\psi}$ is now a K\"ahler metric on $Y$ in the class $[\beta]$, and so $[\alpha]$ is semiample. This last part of the argument is essentially the same as \cite[Remark 3.5]{HP}, \cite[Proof of Theorem 2.7]{TZ}.
\end{proof}

In the proof above we used crucially the assumption that $\lambda[\alpha]-c_1(K_X)$ is nef and big to deduce that the singularities of the normal surface $Y$ were rational. It is instructive to see what happens if $X$ is a compact K\"ahler surface with an arbitrary nef and big $(1,1)$ class $[\alpha]$. Then the first half of the proof of Theorem \ref{big} still applies, and we obtain a K\"ahler current $T=\alpha+\ddbar\psi$ on $X$ which is singular precisely along $E_{nK}(\alpha)=\Null(\alpha)$, and a holomorphic map $\pi:X\to Y$ onto a normal compact complex surface $Y$, so that $\pi$ is an isomorphism away from $\Null(\alpha)$, and contracts each connected component of $\Null(\alpha)$ to a point (the union of these points is denoted by $S\subset Y$). In general now the irreducible components of $\Null(\alpha)$ need not be rational curves.

The pushforward $\pi_*T$ is still a K\"ahler current on $Y$, smooth away from $S$, but in general it need not be locally $\de\db$-exact (unlike the case discussed above in Theorem \ref{big}). Indeed we have the following:

\begin{proposition}
In this setting, if $\pi_*T$ is locally $\de\db$-exact near all points of $S$, then $Y$ is a K\"ahler analytic space and $[\alpha]$ is semiample.
\end{proposition}
\begin{proof}
Indeed, if $T$ was locally $\de\db$-exact near $S$ then it would be locally $\de\db$-exact on all of $Y$, so there is a finite open cover $Y=\cup_\alpha U_\alpha$ with plurisubharmonic functions $u_\alpha$ on $U_\alpha$ such that $\pi_*T=\ddbar u_\alpha$ on $U_\alpha$.
On each nonempty intersection $U_\alpha\cap U_\beta$ we have that $\de\db (u_\alpha-u_\beta)=0$. This implies that $u_\alpha-u_\beta$ is smooth, by regularity of the Laplacian in a local embedding.
Let $\rho_\alpha$ be a partition of unity subordinate to this cover, and let $u=\sum_\alpha \rho_\alpha u_\alpha$. Then $\gamma:=\pi_*T-\ddbar u$ is closed and smooth on $Y$, because on $U_\alpha$ it equals $\ddbar(\sum_\beta\rho_\beta(u_\alpha-u_\beta))$, which is smooth.

Now $\ddbar u=\pi_*T-\gamma$ is smooth on $Y\backslash S$, and so by regularity of the Laplacian we conclude that $u$ is smooth on $Y\backslash S$.
Recall that $\gamma+\ddbar u\geq\ve \omega$ for some $\ve>0$, where $\omega$ is a Hermitian metric on $Y$. Then the same regularized maximum construction as in the proof of Theorem \ref{big} shows that there exists a smooth function $v$ on $Y$, which equals $u$ on $Y$ minus a small neighborhood of $S$, and such that $\gamma+\ddbar v\geq \ve'\omega$ holds on $Y$, for some $\ve'>0$. Therefore $\omega_Y:=\gamma+\ddbar v$ is a K\"ahler metric on $Y$, and $\pi^*\omega_Y$ is a smooth semipositive $(1,1)$ form on $X$.

We conclude the argument as in \cite[Proposition 3.6]{TW}.
Consider $$\eta=T-\pi^*\omega_Y+\ddbar ((v-u)\circ\pi),$$
which is a closed real $(1,1)$ current on $X$, supported on $\Null(\alpha)=\cup_i C_i$, and which is expressed as the difference of two positive currents
(since $\omega_Y+\ddbar (u-v)\geq \ve\omega$).
This last condition implies that its coefficients are measures, and so $\eta$ is a flat current (in the terminology of \cite{Fe}), and Federer's support theorem \cite[4.1.15]{Fe} implies that $\eta=\sum_i\lambda_i[C_i]$ for some
real numbers $\lambda_i$. But integrating $\eta$ over $C_i$ we see that
$$\lambda_i(C_i^2)=\int_{C_i}\eta=\int_{C_i}T=\int_{C_i}\alpha=0,$$ and so $\eta=0$.
Therefore
$$\pi^*\omega_Y=\alpha+\ddbar(\psi+(v-u)\circ\pi),$$
and again by regularity of the Laplacian, we conclude that the function $\vp=\psi+(v-u)\circ\pi$ is smooth on $X$ and $\alpha+\ddbar\vp=\pi^*\omega_Y$, which shows that $[\alpha]$ is semiample.
\end{proof}

There are in fact explicit examples where $Y$ is not a K\"ahler analytic space (and therefore $\pi_*T$ is not locally $\de\db$-exact). Indeed, we can take an example constructed by Grauert \cite[8(d), p.365-366]{Gr} of a projective ruled surface $X$ over a smooth curve of genus at least $2$, which contains a copy $C\subset X$ of this curve with $(C^2)<0$, and with a nef and big line bundle $L$ with $\Null(L)=C$ (we can take for example $L=A-\frac{(A\cdot C)}{(C^2)}C$ where $A$ is any ample line bundle on $X$), such that the contraction $\pi:X\to Y$ of $C$ is not a K\"ahler space, thanks to \cite{Moi}. It also follows that in this case $L$ is not semiample (if it was, $Y$ would be projective). The following problem is therefore very interesting:

\begin{problem}\label{prob}
Prove that in Grauert's example the line bundle $L$ is not Hermitian semipositive.
\end{problem}

This would be the first example of a nef and big line bundle (or $(1,1)$ class) on a projective surface (or compact K\"ahler surface) which is not Hermitian semipositive. Note that the example in \cite{DPS} is not big, and it can be modified \cite[Example 5.4]{BEGZ} to produce nef and big examples but only in dimensions $3$ or higher.

We can now give the proof of Proposition \ref{nef}, which is a well-known classical result:
\begin{proof}[Proof of Proposition \ref{nef}]
The fact that $[\alpha]$ is rational means that $[k\alpha]=c_1(L)$ for some $k\geq 1$ and some holomorphic nef line bundle $L$ on $X$. We will show that $L$ is semiample.

If $\int_X c_1(L)^2>0$ then $L$ is nef and big, so $X$ is projective, and $L$ is semiample by Kawamata-Shokurov's base-point-free theorem. So we may assume that $\int_X c_1(L)^2=0$.

If $c_1(L)=0$ then the exponential sequence gives that $L$ is trivial (since $X$ is simply connected), and in this case the conclusion clearly holds.

So we may assume that $\int_X c_1(L)^2=0$ and that $L$ is not trivial. By Riemann-Roch we have
$$h^0(X,L)+h^0(X,-L)\geq 2,$$
so exactly one among $L$ and $-L$ is effective (since $L$ is nontrivial), and since $L$ is nef we must have $h^0(X,L)\geq 2$. The same argument as in \cite[Theorem 3.8 (b)]{Reid_alg} shows that $L$ is globally generated.
\end{proof}

\section{Rough positive currents}\label{rough}
In this section we give the proof of Theorem \ref{main}, and its refinement in Theorem \ref{thm:main}. First, we introduce some notation and concepts from holomorphic dynamics.

\begin{definition}
	A holomorphic automorphism of a K3 surface $ X $ is called \emph{hyperbolic} if its action on $ H^{1,1}(X,\mathbb{R}) $ has some eigenvalue of norm strictly larger than $ 1 $.
\end{definition}

By the Gromov-Yomdin theorem \cite{Gro, Yo}, this is equivalent to the automorphism having positive topological entropy.

We fix a nowhere-vanishing holomorphic $2$-form $\Omega$ on $X$, and let $\dVol=\Omega\wedge\overline{\Omega}$ be the Lebesgue measure on $X$, which we may assume satisfies $\int_X \dVol=1$.
Because of the uniqueness of $\Omega$ (up to scale) it follows that $\dVol$ is invariant under every automorphism of $X$, and that every Ricci-flat K\"ahler metric $\omega_X$ on $X$ has volume form equal to $\omega_X^2=\left(\int_X\omega_X^2\right) \dVol$.

\medskip
\noindent \textbf{Properties of hyperbolic automorphisms.}
Because the signature of cup-product on $ H^{1,1}(X,\mathbb{R}) $ is $ (1,19) $ the eigenvalues of a hyperbolic automorphism have norm $ \lambda,\lambda^{-1} $ and the other ones are on the unit circle.
The two corresponding eigenclasses (which we normalize up to scaling so that $[\eta_+]\cdot[\eta_-]=1$) contain positive currents $ \eta_+, \eta_- $ with locally H\"older potentials (continuity is proved in \cite[Theorem 3.1]{Cantat_K3} and H\"older continuity in \cite[Proposition 2.4]{DS}, see also \cite[\S 1]{DG}), and their cup-product $ \mu := \eta_+\wedge \eta_- $ is a well-defined probability measure (see \cite{BT}, \cite[\S3.3]{Cantat_K3}), which does not charge any pluripolar subset of $X$ by \cite[Corollary 2.5]{BT} and \cite{BT2}.
The measure $ \mu $ is also called in this context the measure of maximal entropy, as it realizes the topological entropy of the automorphism and any other measure has strictly smaller entropy. Moreover, $ \mu $ is ergodic.

\medskip
\noindent \textbf{Siegel discs.}
If $T:X\to X$ is a hyperbolic automorphism of a $K3$ surface, following \cite{McM} we say that $T$ has a Siegel disc if there is a nonempty open subset $U\subset X$ which is biholomorphic to a polydisc in $\mathbb{C}^2$, $T$ preserves $U$ and $T|_U$ is holomorphically conjugate to an irrational rotation (i.e. of the form $(z_1,z_2)\mapsto (\lambda_1z_1,\lambda_2z_2)$ with $|\lambda_1|=|\lambda_2|=1$ and so that this map has dense orbits on $S^1\times S^1$). The main result of \cite{McM} is that such automorphisms exist, and the $K3$ surfaces $X$ which support them are never projective.

\medskip
\noindent \textbf{Kummer examples.}
\begin{definition}[cf. {\cite[Classification Thm.]{CD}}]
	An automorphism $ T $ of a K3 surface $ X $ is a \emph{Kummer example} if there exists a complex torus $ A $ with an automorphism $ L $, with a map $ X \to A/\pm 1 $ which intertwines the actions of $ T $ and $ L $.
\end{definition}
Note that for a projective Kummer example, we have that the rank of the Picard group is at least $ 17 $ (the $ 16 $ exceptional curves, plus a polarization).

By direct calculation (cf. \cite[pp.30-32]{Cantat_thesis}), if $(X,T)$ is a Kummer example on a $K3$ surface $X$, then both currents $\eta_{\pm}$ (and therefore also the measure $\mu$) are smooth on $X$.

We can now prove the following result:
\begin{theorem}
	\label{thm:main}
	Let $ X $ be a K3 surface, and let $ T\colon X\to X $ be an automorphism which acts hyperbolically in $ H^{1,1}(X,\mathbb{R}) $.
	Then the eigenclasses $ [\eta_+],[\eta_-] \in H^{1,1}(X,\mathbb{R})\setminus \{0\} $ such that $ T^*[\eta_{\pm}] = \lambda^{\pm 1}[\eta_{\pm}] $ have the following properties:
	\begin{itemize}
		\item [(i)] There are unique closed positive currents $ \eta_\pm $ in the corresponding classes. These currents satisfy $T^*\eta_\pm = \lambda^{\pm 1} \eta_\pm$.
        \item [(ii)] Suppose that $T$ has a Siegel disc. Then at least one of the currents $\eta_{\pm}$ cannot be smooth away from a proper closed analytic subset of $X$.
		\item [(iii)] Suppose instead that $X$ is projective and not Kummer. Then the currents $ \eta_{\pm} $ are not smooth at any point of the support of the measure of maximal entropy $ \mu= \eta_+\wedge \eta_- $.
	\end{itemize}
\end{theorem}

Part (iii) uses as input the main result of Cantat--Dupont \cite{CD}.

\begin{proof}[Proof of \autoref{thm:main}(i), Uniqueness]
	
	In the present context this result is due to Cantat \cite[Thm. 2.4]{Cantat_K3}, but the main ideas of the proof trace back to the work of Forn\ae ss-Sibony \cite{FS} on H\'enon maps in $\mathbb{C}^2$, see also the work of Dinh-Sibony \cite{DS,DS2} for much stronger and general results. We reproduce the short proof for completeness.
	
 	To construct $\eta_\pm$ fix any \Kahler metric $ \omega $ and consider the sequence
	\begin{align}
	 \eta_N := \frac 1 N\sum_{i=1}^N \frac{(T^{*})^i} {\lambda^i} \omega \quad
	 \textrm{so that} \quad
	(T^*-\lambda)\eta_N = \frac 1 N \left( \frac{(T^{*})^{N+1}}{\lambda^N} - T^* \right) \omega
	\end{align}
	Pick some weak limit of $ \eta_N $, denoted $ \eta_+ $; then we have $ \eta_+\geq 0 $ and $ T^*\eta_+ = \lambda \eta_+ $. The same construction applied to $T^{-1}$ gives $\eta_-$. If necessary, we then scale $\eta_\pm$ so that $[\eta_+]\cdot[\eta_-]=1$.
	
	For uniqueness, it is easier to show it for the class $ [\eta_-] $; for $ \eta_+ $ the argument is the same but applied with $ T^{-1} $.
    Fix a smooth representative $\gamma$ of $[\eta_-]$, and let $ \alpha \geq 0 $ be a closed positive current in the same cohomology class $ [\eta_-] $.
	Then we can write $\alpha=\gamma+dd^c u_1$ and $\eta_-=\gamma+dd^c u_2$, for some quasi-psh functions $u_1,u_2$ which we may normalize with $\int_X u_j \dVol=0, j=1,2$. It is well-known that $u_1$ and $u_2$ are in $L^1$, and in fact for any quasi-psh function $v$ with $\gamma+dd^c v\geq 0$ and $\int_X v\dVol=0$ we have the estimate
    \[\norm{v}_{L^1}:= \int_X |v| \dVol\leq C_0,\]
    where the constant $C_0$ depends only on $(X,g)$ (the geometry of the background K\"ahler metric) and $\gamma$, but is independent of $v$. See for example \cite[Proposition 2.1]{DK} (there $\gamma$ is assumed to be positive definite, but this is not used in the proof; also they use the normalization $\sup_X v=0$, but these give uniformly equivalent constants). Therefore we see that every quasi-psh function $v$ with $\eta_-+dd^c v\geq 0$ and $\int_X v\dVol=0$ satisfies $\int_X |v|\dVol\leq 2C_0$.
    Then if we let $u=u_1-u_2$ we have $ \alpha = \eta_- + dd^c u,$ together with $ \int_X u \dVol =0 $ and $\int_X |u|\dVol\leq 2C_0$.
		
	We also have $$0\leq \lambda T^*\alpha = \lambda T^*\eta_- + dd^c (\lambda T^*u) =\eta_-+dd^c(\lambda T^*u),$$ and since $T$ preserves $\dVol$, we still have
	$\int_X (\lambda T^*u)\dVol=0$, and so
	\[2C_0\geq \lambda \norm{T^*u}_{L^1}=\lambda \norm{u}_{L^1}.\]
	
	Iterating the map $ T $ indefinitely and applying the same argument, we obtain $ \norm{u}_{L^1} = 0 $ and hence $ u=0 $.

Lastly, note that uniqueness immediately implies that $ T^*\eta_\pm = \lambda^{\pm 1} \eta_\pm $.
\end{proof}

\begin{proof}[Proof of \autoref{thm:main}(ii), Non-smoothness with Siegel disc]
If both $\eta_{\pm}$ are smooth away from proper closed analytic subsets of $X$, it follows that there is a proper closed analytic subset $E\subset X$ such that the restriction of $\mu=\eta_+\wedge\eta_-$ to $X\backslash E$ can be written as $\mu|_{X\backslash E}= f \dVol|_{X\backslash E}$ for some smooth nonnegative function $f$ on $X\backslash E$.
Moreover $f$ is not identically zero since $\mu$ does not charge pluripolar sets.

Since $\mu$ is ergodic and $\dVol$ is $T$-invariant, it follows that $f$ is constant $\mu$-a.e. on $X\backslash E$ (i.e. on the set where $f$ is positive).
Since $f$ is also smooth on $X\backslash E$, it follows that it is constant and in fact $f\equiv 1$ since both $\mu$ and $\dVol$ are probability measures on $X$ which do not charge $E$.
This shows that $\mu=\dVol$ as measures on $X$.
However, as observed in \cite[Theorem 11.2]{McM}, on the Siegel disc $U\subset X$ we must have $\eta_{\pm}|_U=0$, and therefore $\mu|_U=0$ too, which is a contradiction to the fact that $U$ has positive Lebesgue measure.
\end{proof}
In particular, the fact that $\mu|_U=0$ implies that $X\backslash U$ is not pluripolar, since $\mu$ is a probability measure which does not charge pluripolar sets. It is however not clear to us whether $X\backslash U$ must have positive Lebesgue measure.

Before continuing with the proof of \autoref{thm:main}(iii) we recall some preliminaries from dynamics.

\medskip

\noindent \textbf{Dimension, Entropy, Lyapunov exponents.}
The main reference for the following results is the work of Ledrappier--Young \cite{LY2}, which was preceded by many earlier results, including those of Ruelle, Margulis, Ma\~n\'e.
The results are stated for a holomorphic automorphism $ T $ of a K3 surface, so the Lyapunov exponents are always of the form $ \lambda, -\lambda $ and each occurs with multiplicity $ 2 $.
For any ergodic $ T $-invariant measure $ \mu $ we have
\begin{align}
h(\mu) = \lambda(\mu) \cdot \dim_+(\mu)
\end{align}
where $ h(\mu) $ is the entropy, $ \lambda(\mu) $ is the (positive) Lyapunov exponent and $ \dim_+(\mu) $ is the dimension of $ \mu $ along the unstable foliation.

To define dimension precisely, note that for $ \mu $-a.e. $ x $ there will be a (local) unstable manifold $ W^+(x) $ and a family of measures $ \mu_+(x) $ defined on $ W^+(x) $; the $ \mu_+(x) $ are disintegrations of $ \mu $ along the unstable directions.
Denoting by $ B(x,r) $ the ball of radius $ r $ in $ W^+(x) $ (in the induced metric from $X$), the limit
$$\dim_+(\mu):=\lim_{r\to 0} \frac{\log \mu_+(x)(B(x,r))}{\log r },$$ will exist $ \mu $-a.e. and will be called the dimension of $ \mu $ along unstables.

Finally, the main result of Ledrappier--Young \cite{LY1} in this case implies that $ \dim_+(\mu) = 2$ if and only if $ \mu $ is absolutely continuous with respect to Lebesgue measure.

\begin{proof}[Proof of \autoref{thm:main}(iii), Non-smoothness for non-Kummer]
	Suppose now that $ \eta_+ $ is a stable current of a hyperbolic automorphism of a projective K3 surface, which is not a Kummer example.
	Let $ \mu $ be the measure of maximal entropy and $ \dVol=\Omega\wedge\overline{\Omega} $ be the natural invariant Lebesgue measure coming from the holomorphic $ 2 $-form $\Omega$.
	The main result of Cantat--Dupont \cite{CD} implies that $ \mu $ is not absolutely continuous with respect to $ \dVol $.
	Thus $ \dim_+(\mu)<2 $, by the remarks above.
	
	Next, the discussion in \cite[pg. 42-43]{Cantat_K3} (see particularly ``Conclusion'' at the bottom of pg. 43) implies that the unstable measures $ \mu_+(x) $ can be described as follows, in a neighborhood of $ x $ in the unstable manifold $ W^+(x) $.
	The unstable manifold $ W^+(x) $ is an immersed holomorphic curve in the K3 surface $ X $, mapped as $ i:W^+(x)\to X $ (see e.g. \cite[\S2.6]{BLS}).
	Restrict the immersion $ i $ to a small neighborhood of $ x\in W^+(x) $ and pull back the current $ \eta_+ $ to obtain $ i^*\eta_+ =: \mu_+(x) $, which is denoted $ W^+(x)\wedge \eta_+ $ in \cite{Cantat_K3}. Note that the roles of $\eta_+$ and $\eta_-$ are exchanged in \cite{Cantat_K3}, since there he uses the convention $T_*\eta_{\pm}=\lambda^{\pm 1}\eta_{\pm}$ instead of the one that we use ($T^*\eta_{\pm}=\lambda^{\pm 1}\eta_{\pm}$), following \cite{CD,McM}.
	
	Because the dimension of $ \mu_+(x) $ is strictly less than $ 2 $, it follows that $ \eta_+ $ cannot be even continuous in a neighborhood of $ x $, for $ \mu $-a.e. $ x $.
\end{proof}

\begin{remark}
Thanks to the result in Theorem \ref{thm:main} (i), the classes $[\eta_{\pm}]$ have the remarkable rigidity property that they contain a unique closed positive current (such classes are called {\em rigid} in \cite{DS2}). Other examples of pseudoeffective classes with this property are all those classes which are equal to their negative part in Boucksom's divisorial Zariski decomposition \cite[Proposition 3.13]{Bou}, as well as the nef $(1,1)$ class constructed in \cite{DPS} which is not semipositive. However, in all these other examples the unique closed positive current has positive Lelong numbers (in fact, it is the current of integration along an effective $\mathbb{R}$-divisor), while in the setting of Theorem \ref{thm:main} the Lelong numbers vanish, since as mentioned earlier the current even has H\"older continuous local potentials.
\end{remark}

\begin{remark}\label{refe}
The referee kindly points out that there are other known examples of a rigid nef class which even contains a smooth semipositive representative. For this, it suffices to find a surface $X$ with a nonsingular holomorphic foliation which admits a unique invariant transverse probability measure which is further smooth. In this case if $\theta$ is the smooth laminar current of the foliation then $\int_X\theta^2=0$ and so if $T\in[\theta]$ is any closed positive current then $T\wedge\theta=0$. Hence $T$ is also laminar for the foliation, and by uniqueness of the invariant transverse measure we must have $T=\theta$. An explicit example was constructed by Mumford \cite[Example 1.5.1]{Laz}, with $X$ a ruled surface over a curve of genus $2$ and $[\theta]=c_1(L)$ for $L$ nef with trivial section ring.
\end{remark}
\begin{remark}
In fact, as the proofs above show, in the statements of both Theorem \ref{thm:main} (ii) and (iii) we can in fact replace the word ``smooth'' with ``continuous''.
\end{remark}

Recall the following questions which were raised by the second-named author:

\begin{question}[Question 5.5 in \cite{To4}]\label{q1}
Let $X^n$ be a compact Calabi-Yau manifold and $[\alpha_0]$ a nef $(1,1)$ class on $X$. Then there is a smooth closed semipositive real $(1,1)$ form $\omega_0$ on $X$ in the class $[\alpha_0]$.
\end{question}

The answer to this question is affirmative when $X$ is projective, $[\alpha_0]$ is nef and big and belongs to the real N\'eron-Severi group, by the real version of the base-point-free Theorem (see e.g. \cite[Theorem 2.3]{To}).

\begin{question}[Question 4.3 in \cite{To3}, and Question 5.3 in \cite{To4}] \label{q2}
Let $(X^n,\omega_X)$ be a compact Ricci-flat Calabi-Yau manifold, $[\alpha_0]$ a nef $(1,1)$ class with $\int_X\alpha_0^n=0$, and for $t>0$ let $\omega_t$ be the unique Ricci-flat K\"ahler metric on $X$ cohomologous to $[\alpha_0]+t[\omega_X]$. Then there is a proper closed analytic subvariety $V\subset X$ and a smooth closed semipositive real $(1,1)$ form $\omega_0$ on $X\backslash V$ with $\omega_0^n=0$ such that $\omega_t\to\omega_0$ in $C^\infty_{\rm loc}(X\backslash V)$ as $t\to 0$.
\end{question}

This question has an affirmative answer when $[\alpha_0]$ is semiample (so it is the pullback of a K\"ahler class under a map $f:X\to Y$) and the generic fiber of $f$ is a torus, by the results in \cite{GTZ,HT}. If $[\alpha_0]$ is semiample (but the generic fiber of $f$ is not necessarily a torus), then this question is answered affirmatively in \cite{TWY}, except that the convergence $\omega_t\to \omega_0$ is only known to happen in $C^0_{\rm loc}(X\backslash V)$.

However, despite these positive results, in general we have the following:

\begin{theorem}\label{negative}Questions \ref{q1} and \ref{q2} have a negative answer.
\end{theorem}
\begin{proof}
Let us take $X$ to be a non-projective $K3$ surface with an automorphism $T:X\to X$ with positive entropy and with a Siegel disc, constructed by McMullen \cite{McM}. Then from Theorem \ref{thm:main} (ii) we immediately see that at least one among the two nef classes $[\eta_{\pm}]$ (say $[\eta_+]$) has no smooth semipositive representative (indeed the only closed positive current in that class is $\eta_+$ which is not even smooth on a Zariski open subset), thus answering Question \ref{q1} negatively.

For Question \ref{q2}, take $[\alpha_0]=[\eta_+]$, so that for every $t>0$ the classes $[\alpha_0]+t[\omega_X]$ are K\"ahler, and they each contain a unique Ricci-flat K\"ahler metric $\omega_t$. By weak compactness of currents, given any sequence $t_i\to 0$, up to passing to a subsequence, the metrics $\omega_{t_i}$ converges weakly as currents to a closed positive current in the class $[\eta_+]$ which by Theorem \ref{thm:main} (i) must equal $\eta_+$. In particular, $\omega_t\to\eta_+$ as currents as $t\to 0$. If there was a proper closed analytic subvariety $V\subset X$ such that we had smooth convergence on compact subsets of $X\backslash V$, this would imply that $\eta_+$ is smooth on $X\backslash V$, contradicting Theorem \ref{thm:main} (ii).
\end{proof}

\begin{remark}
In these examples with a Siegel disc $U\subset X$, let $\omega_X$ be a fixed Ricci-flat K\"ahler metric on $X$ and for $i\geq 0$ consider the Ricci-flat K\"ahler metrics
$$\omega_i=\frac{(T^*)^i\omega_X}{\lambda^i}.$$
Since their cohomology classes $[\omega_i]$ converge to $[\eta_+]$ as $i\to\infty$, it follows as above from Theorem \ref{thm:main} (i) that the metrics $\omega_i$ converge to $\eta_+$ weakly as currents on $X$. By the argument we just did, the metrics $\omega_i$ cannot converge smoothly on any Zariski open subset of $X$. On the other hand, it is observed in \cite[Theorem 11.2]{McM} that $\omega_i|_U$ converge smoothly to the zero form.
\end{remark}

To conclude, we give the proof of Theorem \ref{main}:
\begin{proof}[Proof of Theorem \ref{main}]
Part (a) follows immediately from Theorem \ref{thm:main} (ii), together with McMullen's examples \cite{McM} of hyperbolic $K3$ automorphisms with a Siegel disc.

For part (b), let $X$ be a generic hypersurface in $\mathbb{P}^1\times\mathbb{P}^1\times\mathbb{P}^1$ of degree $(2,2,2)$ with $T$ the composition of the three involutions of $X$ obtained by expressing $X$ as a ramified double cover of $\mathbb{P}^1\times\mathbb{P}^1$ in three different ways and interchanging the $2$ sheets of each cover.
It is proved in \cite{Cantat_K3,McM} that $T$ has positive entropy.
This $K3$ surface is not Kummer because the rank of the Picard group of a generic $ (2,2,2) $ surface will be $ 3 $, generated by the pullbacks of the hyperplane bundles on $\mathbb{P}^1\times\mathbb{P}^1$ under the three double covering maps (see e.g. \cite[p. 37]{Cantat_thesis}), whereas for projective Kummer $K3$ surfaces the rank of Picard has to be at least $ 17 $.

On $X$ the set of $(1,1)$ classes $[\alpha]$ in the real N\'eron-Severi group which satisfy $\int_X\alpha^2=1$ (equipped with the intersection pairing) is isometric to the Poincar\'e disc, with $[\eta_+]$ and $[\eta_-]$ lying on its ideal boundary, see e.g. \cite[p.9]{Cantat_K3}, \cite[pp.36-37]{Cantat_thesis}. In particular, there are nef $\mathbb{R}$-divisors $D_{\pm}$ such that $c_1(D_{\pm})=[\eta_{\pm}]$. By Theorem \ref{thm:main} (iii) none of these two $\mathbb{R}$-divisor can be Hermitian semipositive.
\end{proof}

%====================================================
%====================================================
%		Bibliography
%====================================================
%====================================================
\bibliographystyle{bib_style}
\bibliography{rough_positive_currents}
%====================================================

\end{document}